\documentclass[11pt]{article}
\usepackage{amsmath, amsfonts, amsthm, amssymb, color}
\usepackage{graphicx}
\usepackage{float}
\usepackage{verbatim}

\usepackage{subcaption}

\usepackage{bbm}

\allowdisplaybreaks

\numberwithin{equation}{section}

\newtheorem{lemma}{Lemma}[section]

\newtheorem{thm}[lemma]{Theorem}
\newtheorem{cor}[lemma]{Corollary}
\theoremstyle{definition}

\newtheorem{conjecture}[lemma]{Conjecture}

\theoremstyle{remark}

\def\O{\mathcal{O}}

\numberwithin{equation}{section} \numberwithin{table}{section}

\title{Approximating elements of the middle third Cantor set with dyadic rationals}
\author{Simon Baker\\ \\
\emph{School of Mathematics,} \\ \emph{University of Birmingham,} \\ \emph{Birmingham,  B15 2TT, UK.} \\ Email: simonbaker412@gmail.com\\}

\date{\today}
\begin{document}
\maketitle

\begin{abstract}
Let $C$ be the middle third Cantor set and $\mu$ be the $\frac{\log 2}{\log 3}$-dimensional Hausdorff measure restricted to $C$. In this paper we study approximations of elements of $C$ by dyadic rationals. Our main result implies that for $\mu$ almost every $x\in C$ we have
$$\#\left\{1\leq n\leq N:\left|x-\frac{p}{2^n}\right| \leq \frac{1}{n^{0.01}\cdot 2^{n}}\textrm{ for some }p\in\mathbb{N}\right\}\sim 2\sum_{n=1}^{N}n^{-0.01}.$$ This improves upon a recent result of Allen, Chow, and Yu which gives a sub-logarithmic improvement over the trivial approximation rate.\\


\end{abstract}

\section{Introduction}
Diophantine approximation is the study of how well real numbers can be approximated by rational numbers. An interesting and well studied problem in this area is to determine  how well elements of a fractal set can be approximated by rational numbers. This problem has its origins in a paper of Mahler \cite{Mah} who posed the following two questions:
\begin{enumerate}
	\item How well can elements of the middle third Cantor set be approximated by rational numbers in the middle third Cantor set?
	\item How well can elements of the middle third Cantor set be approximated by rational numbers outside of the middle third Cantor set?
\end{enumerate} 
These two questions have generated a substantial amount of research (see \cite{ACY,Bak,BFR,Bug2,BugDur,FMS,FS2,FS,KL,KLW,LSV,Sch,Shp,SW,TWW,Weiss,Yu} and the references therein). We do not attempt to give an exhaustive overview of research in this area. We instead detail a few key papers that are particularly relevant to this work, and outline some important recent results. In \cite{LSV} Levesley, Salp, and Velani made significant progress with the first of Mahler's questions. They considered the following setup: Let $C$ denote the middle third Cantor set and $\mu$ be the $\frac{\log 2}{\log 3}$-dimensional Hausdorff measure restricted to $C$. Given $\psi:\mathbb{N}\to [0,\infty)$ we can associate the set
$$W_{3}(\psi):=\left\{x\in C:\left|x-\frac{p}{3^n}\right|\leq \frac{\psi(n)}{3^n}\textrm{ for infinitely many } (p,n)\in\mathbb{N}\times\mathbb{N}\right\}.$$
Theorem 1 of \cite{LSV} provides a simple criteria for determining $\mu(W_{3}(\psi)).$ This criteria is phrased in terms of the convergence/divergence of naturally occurring volume sums. As such, in the setting of the sets $W_{3}(\psi)$ this result provides a natural analogue of a well known theorem due to Khintchine \cite{Khi}. What prevents the work of Levesley, Salp and Velani from providing a complete solution to the first of Mahler's questions is the existence of rational numbers in $C$ whose base three expansion is infinite and eventually periodic. For such a rational number it is a challenging problem to determine what, if any, cancellation occurs between the numerator and the denominator when we express it using geometric series. This makes determining good rational approximations more difficult. This problem was considered in \cite{Bak} and \cite{TWW}. These papers establish an intrinsic analogue of Khintchine's theorem for self-similar measures with respect to a non-standard height function. This height function behaves well with respect to the underlying iterated function system\footnote{We refer the reader to \cite{Fal} for background on iterated function systems, self-similar measures, and self-similar sets.}. In \cite{SW} Simmons and Weiss proved that any non-atomic self-similar measure on $\mathbb{R}$ gives zero mass to the set of  badly approximable numbers. They in fact proved the stronger statement that for one of these self-similar measures, the orbit of a typical point under the action of the Gauss map equidistributes with respect to the Gauss measure. This result in particular applies to our measure $\mu$ which can be identified with a self-similar measure. In a recent breakthrough result, Khalil and Luethi \cite{KL} gave necessary conditions for a self-similar measure so that a complete analogue of Khintchine's theorem is satisfied. This result however requires the self-similar measure to be of sufficiently large dimension and in fact does not apply to $\mu$.

Motivated by the questions of Mahler and the works detailed above, in this paper we study rational approximations of elements of $C$ by dyadic rationals, i.e. rational numbers of the form $p/2^n$. Given $\psi:\mathbb{N}\to [0,\infty)$ we associate the set $$W_{2}(\psi):=\left\{x\in C:\left|x-\frac{p}{2^n}\right|\leq \frac{\psi(n)}{2^n}\textrm{ for infinitely many } (p,n)\in\mathbb{N}\times\mathbb{N}\right\}.$$ $W_{2}(\psi)$ has the following simple shrinking target reinterpretation. Let $T_{2}:\mathbb{R}/\mathbb{Z}\to\mathbb{R}/\mathbb{Z}$ be given by $T_{2}(x)=2x \mod 1$. Then we have $$W_{2}(\psi)=\left\{x\in C:d(T_{2}^{n}(x),0)\leq \psi(n)\textrm{ for infinitely many } n \in \mathbb{N}\right\}.$$ Here $d$ is the standard metric on $\mathbb{R}/\mathbb{Z}.$ With this shrinking target reinterpretation in mind, there is a natural inhomogeneous generalisation of $W_{2}(\psi)$. Given a sequence $\mathbf{x}=(x_n)_{n=1}^{\infty}\in(\mathbb{R}/\mathbb{Z})^{\mathbb{N}},$ we define $$W_{2}(\psi,\mathbf{x}):= \left\{x\in C:d(T_{2}^{n}(x),x_n)\leq \psi(n)\textrm{ for infinitely many } n\in \mathbb{N}\right\}.$$ The sets $W_{2}(\psi)$ and $W_{2}(\psi,\mathbf{x})$ will be the main focus of this paper. Our motivation for studying these sets comes not only from Mahler's questions, but also from the well known Furstenberg $\times 2$ $\times 3$ principle from Ergodic Theory. Loosely speaking, this principle states that a mathematical object cannot exhibit exceptional dynamical behaviour for both the map $T_2$ and the map $T_3$ ($T_3:\mathbb{R}/\mathbb{Z}\to \mathbb{R}/\mathbb{Z}$ given by $T_{3}(x)=3x \mod 1$). The first instance of this phenomenon being verified was in a paper of Furstenberg \cite{Fur}. He proved that if $X$ is a closed infinite subset of $\mathbb{R}/\mathbb{Z}$ that is invariant under $T_{2}$ and $T_{3}$ then $X=\mathbb{R}/\mathbb{Z}.$ This result gave rise to the so called  $\times 2$ $\times 3$ conjecture which aims to establish a measure theoretic analogue of this result. Significant progress was made towards this conjecture by Rudolph in \cite{Rud}. He proved that if $\nu$ is a probability measure that is invariant and ergodic for $T_{2}$ and $T_{3},$ then either $\nu$ has zero entropy for both $T_2$ and $T_3$ or $\nu$ is the Lebesgue measure \cite{Rud}. See \cite{Joh} for a further generalisation of this result. In \cite{Host} Host obtained the following strengthening of Rudolph's result. He proved that if $\nu$ is invariant and ergodic for $T_{2}$ and has positive entropy, then the orbit of $\nu$ almost every $x$ under $T_{3}$ equidistributes with respect to the Lebesgue measure. Note that this result still holds with the roles of $T_2$ and $T_3$ reversed. Further generalisations of Host's theorem were obtained in \cite{HocShm} and \cite{Lin}. See also \cite{Hoc} for a short proof of this theorem. We also refer the reader to \cite{ABS} and \cite{HocShm} where analogues of Host's theorem were established for self-similar measures. Further important results on the sum and intersection of a $T_{2}$-invariant set and a $T_{3}$-invariant set were obtained in \cite{HocShm2,Shmerkin,Wu}.

Despite the tremendous amount of activity that has been undertaken and is still ongoing to understand the $\times 2$ $\times 3$ principle, until recently no work had appeared which considered this principle in the context of shrinking targets problems. For the purpose of formulating a meaningful $\times 2$ $\times 3$ conjecture in the setting of shrinking targets, we recall the following well known result.
\begin{thm}
	\label{babyshrinking}
	 Let $\nu$ denote the Lebesgue measure restricted to $\mathbb{R}/\mathbb{Z}$. Then for any $\psi:\mathbb{N}\to [0,\infty)$ and $(x_n)\in(\mathbb{R}/\mathbb{Z})^{\mathbb{N}}$ we have $$\nu\Big(\left\{x\in\mathbb{R}/\mathbb{Z}:d(T_{2}^n(x),x_n)\leq \psi(n)\textrm{ for infinitely many }n\in\mathbb{N}\right\}\Big)=
	\begin{cases}
	0 \text{ if }\sum_{n=1}^{\infty}\psi(n)<\infty\\
	1 \text{ if }\sum_{n=1}^{\infty}\psi(n)=\infty.
	\end{cases}
	$$
\end{thm}
For a proof of Theorem \ref{babyshrinking} we refer the reader to \cite{Phi} where an additional asymptotic for the number of solutions is also obtained. With the above results in mind, in particular the theorem of Host, the following conjecture is natural.
\begin{conjecture}
	\label{Velani conjecture}
Let $\psi:\mathbb{N}\to [0,\infty)$ and $\mathbf{x}\in(\mathbb{R}/\mathbb{Z})^{\mathbb{N}}$, then
$$\mu(W_{2}(\psi,\mathbf{x}))=
\begin{cases}
	0 \text{ if }\sum_{n=1}^{\infty}\psi(n)<\infty\\
	1 \text{ if }\sum_{n=1}^{\infty}\psi(n)=\infty.
\end{cases}
$$
\end{conjecture} 
Conjecture \ref{Velani conjecture} predicts that a $\mu$ typical point will exhibit the same behaviour as a Lebesgue typical point. In the special case where $\psi$ is monotonic and $\mathbf{x}$ takes the constant value zero, Conjecture \ref{Velani conjecture} is attributed to Velani in \cite{ACY}. We expect this conjecture to be true in this more general framework so formulate it this way.

The first paper to consider Conjecture \ref{Velani conjecture} and the sets $W_{2}(\psi)$ was a work by Allen, Chow, and Yu \cite{ACY}. Their main results are summarised below.

\begin{thm}[\cite{ACY}, Theorem 1.5]
	\label{Convergencethm}
If $$\sum_{n=1}^{\infty}\left(2^{-\log n/\log \log n \cdot \log \log \log n}\psi(n)^{\log 2/\log 3}+\psi(n)\right)<\infty$$ then $\mu(W_{2}(\psi))=0.$
\end{thm}

\begin{thm}[\cite{ACY}, Theorem 1.9]
	\label{Divergencethm}
For $\psi(n)=2^{-\log \log n/\log \log \log n},$ we have $\mu(W_{2}(\psi))=1$.
\end{thm}
For the purposes of this paper the more relevant result is Theorem \ref{Divergencethm}. This result gives a sub-logarithmic improvement over the trivial approximation function given by $\psi(n)= 1$ for all $n$. For this trivial approximation function we obviously have $W_{2}(\psi)=C.$ The main result of this paper is the following theorem.
\begin{thm}
	\label{Main thm}
Let $\mathbf{x}\in(\mathbb{R}/\mathbb{Z})^{\mathbb{N}}.$ Then for $\mu$ almost every $x\in C$ we have  $$\#\left\{1\leq n\leq  N:d(T_{2}^n(x),x_n) \leq \frac{1}{n^{0.01}}\right\}\sim 2\sum_{n=1}^{N}n^{-0.01}.$$ 
\end{thm}

If in Theorem \ref{Main thm} we take $\mathbf{x}$ to be the constant sequence that always takes the value $0$, we obtain the following improvement over the full measure result of Allen, Chow, and Yu. We formulate this result in terms of rational approximations by dyadic rationals.

\begin{cor}
	\label{Cor2}
For $\mu$ almost every $x\in C$ we have  $$\#\left\{1\leq n\leq  N:\left|x-\frac{p}{2^n}\right| \leq \frac{1}{n^{0.01}\cdot 2^{n}}\textrm{ for some }p\in\mathbb{N}\right\}\sim 2\sum_{n=1}^{N}n^{-0.01}.$$
\end{cor}
Corollary \ref{Cor2} improves upon Theorem \ref{Divergencethm} in two ways. First of all it establishes a polynomial approximation rate as opposed to a sub-logarithmic approximation rate. It also gives an asymptotic for the number of solutions to these Diophantine inequalities. 

Theorem \ref{Main thm} also has the following nice consequence for the complexity of the dyadic expansion of a $\mu$ typical point.
\begin{cor}
	\label{Cor}
	For $\mu$-almost every $x$ its dyadic expansion $(a_n)_{n=1}^{\infty}\in\{0,1\}^{\mathbb{N}}$ satisfies $a_{n+1}\ldots a_{n+\lfloor 0.01\log_{2}n\rfloor}=\underbrace{0\ldots 0}_{\lfloor 0.01\log_{2}n\rfloor}$ for infinitely many $n\in\mathbb{N}$.
\end{cor}
Corollary \ref{Cor} doesn't formally follows from Theorem \ref{Main thm}. However, we can easily adapt our proof to show that for any $c>0$ Theorem \ref{Main thm} still holds with $n^{-0.01}$ replaced by $c\cdot n^{-0.01}$. It is then relatively straightforward to see that this stronger result implies Corollary \ref{Cor}.

We will prove Theorem \ref{Main thm} in the next section. We end this introductory section by introducing some notation that we will use throughout.\\

\noindent \textbf{Notation.} Let $f$ and $g$ be two complex valued functions defined on a set $S$. We say $f\ll g$ or $f=\O(g)$ if there exists $C>0$ such that $|f(x)|\leq C |g(x)|$ for all $x\in S$. We write $f\asymp g$ if $f\ll g$ and $g\ll f$. We also let $e(x)=e^{2\pi i x}$ for $x\in\mathbb{R}$. 

\section{Proof of Theorem \ref{Main thm}}
We split our proof of Theorem \ref{Main thm} into the following three sections. In the first section we recall an inequality of Jackson on uniform convergence of Fourier series, and we collect some useful estimates on the Fourier transform of $\mu$. In the second section we use these results to obtain certain integral estimates. In the final section we use these integral estimates to prove Theorem \ref{Main thm}. 
\subsection{Preliminary lemmas}
Given a $\mathbb{Z}$ periodic function $f:\mathbb{R}\to\mathbb{R}$ for which $\int_{0}^{1}|f(x)|\, dx<\infty$, the Fourier coefficients of $f$ are defined by the formula 
\begin{equation*}
\label{Fouriertransform2}
\hat{f}(l):=\int_{0}^{1}f(x)e^{-2\pi i l x}\, dx.
\end{equation*}  The following inequality due to Jackson gives a uniform convergence rate for the Fourier series of a $C^{1}$ function (see \cite[Thm. 1.IV]{Jac}).
\begin{lemma}
	\label{Jackson's inequality}
Let $f:\mathbb{R}\to\mathbb{R}$ be a $\mathbb{Z}$ periodic $C^{1}$ function whose derivative has a modulus of continuity $\omega$. Then $$\left|f(x)-\sum_{l=-N}^{N}\hat{f}(l)e^{2\pi i lx}\right|\leq K\frac{\log N}{N}\omega(1/ N)$$ for all $x\in \mathbb{R}$. Here $K>0$ is a constant that does not depend upon $f$ or $N$.
\end{lemma}

The key to proving Theorem \ref{Main thm} is the following lemma. It implies that the Fourier transform of $\mu$ decays to zero polynomially fast outside of a relatively small set of exceptions. 

\begin{lemma}
	\label{Casselslemma}
	Let $N\in\mathbb{N}$ and $l\in\mathbb{Z}\setminus \{0\}$. Then there exists $C_{1},C_{2}>0$ independent of $N$ and $l$ such that 
	$$\# \left\{0\leq n<N:\left|\int e(l2^nx)\, d\mu\right|> C_{1}N^{-0.078}\right\}\leq C_{2}N^{0.922}.$$
\end{lemma}

Lemma \ref{Casselslemma} essentially follows from Cassels' proof of Lemma 2 from \cite{Cas}. However, because this lemma is not formulated in the language of Lemma \ref{Casselslemma}, and the specific constants are not given, we include a proof.

\begin{proof}
We start our proof by recalling the following well known identity for the Fourier transform of $\mu$\footnote{This identity is most commonly stated in the context of the Fourier transform of the $(1/2,1/2)$ self-similar measure on $C$. However, because this self-similar measure coincides with the $\frac{\log 2}{\log 3}$-dimensional Hausdorff measure restricted to $C$ there is no issue here.}:
\begin{equation*}
\left|\int e(lx)\,d\mu\right|=\prod_{k=1}^{\infty}|\cos(l\pi/3^{k})|.
\end{equation*}
Now let us fix $N\in\mathbb{N}$ and $l\in\mathbb{Z}\setminus \{0\}$. For even $n$ we have that $$\int e(l2^nx)\, d\mu=\int e(l4^{n/2}x)\, d\mu,$$ and for odd $n$ we have that $$\int e(l2^nx)\, d\mu=\int e(2l4^{(n-1)/2}x)\, d\mu.$$ Therefore to prove our result it suffices to show that there exists $C_{1},C_{2},C_{3},C_{4}>0$ that do not depend upon $N$ or $l$ such that 
\begin{equation}
\label{WTS1}
\# \left\{0\leq n<N/2:\left|\int e(l4^nx)\, d\mu\right|> C_{1}N^{-0.078}\right\}\leq C_2N^{0.922},
\end{equation}
and 
\begin{equation}
\label{WTS2}
\# \left\{0\leq n<N/2:\left|\int e(2l4^nx)\, d\mu\right|> C_{3}N^{-0.078}\right\}\leq C_4N^{0.922}.
\end{equation}  As we will see below, the benefit of rephrasing our problem in terms of powers of $4$ is that they have useful properties when considered modulo powers of $3$. The proofs of \eqref{WTS1} and \eqref{WTS2} are analogous so we only give the details for \eqref{WTS1}.

We start our proof of \eqref{WTS1} by defining $r\in\mathbb{N}$ via the inequalities 
\begin{equation}
\label{r equation}
3^{r-1}< N/2\leq 3^{r}.
\end{equation} Importantly $4$ has the property that $4 \equiv 1 \mod 3$ and $4\not\equiv 1 \mod 9$. Then as is asserted by Cassels in \cite{Cas}, it is a consequence of this property that the sequence $(4^{n})_{0\leq n<3^r}$ runs modulo $3^{r+1}$ through all residue classes which are congruent to $1$ modulo $3$ (for a detailed proof of this fact see \cite[Proof of Lemma 6.2]{Bug3}). We write $l=3^{m}l'$ where $l'$ is coprime to $3$. Then the sequence $(l4^{n})_{0\leq n<3^r}$ when taken modulo $3^{r+m+1}$ runs through all residue classes that are congruent to $l$ modulo $3^{m+1}$. This implies that if we consider the base three expansions of our numbers of the form $l4^{n}$, i.e. the sequences of zeros, ones, and twos for which $l4^{n}=\sum_{j=0}^{\infty}a_{j}(n)3^{j},$ then for any $b_{1}\ldots b_{r}\in\{0,1,2\}^{r}$ there exists a unique $0\leq n<3^r$ such that the base three expansion of $l4^{n}$ satisfies $a_{m+1}(n)\ldots a_{m+r}(n)=b_{1}\ldots b_{r}.$ Put more succinctly, we have 
\begin{equation}
\label{Bijection}
\{a_{m+1}(n)\ldots a_{m+r}(n):0\leq n<3^r\}=\{0,1,2\}^{r}.
\end{equation}
We now split our parameter space of $n$ into two sets:
$$S_{1}:=\left\{0\leq n<3^{r}:\#\{m+1\leq j\leq m+r:a_{j}(n)=1\}\geq  \frac{r}{8}\right\}$$ and $$S_{2}:=\{0\leq n<3^r\}\setminus S_{1}.$$ For each value of $j$ for which $a_{j}(n)=1,$ there exists a value of $k$ that can be chosen uniquely such that $\cos(l4^n\pi/3^{k})$ takes values in $[-1/2,1/2]$. Therefore for $n\in S_{1}$ we have
$$\left|\int e(l4^nx)\, d\mu\right|=\prod_{k=1}^{\infty}|\cos(l4^n\pi/3^k)|\leq (1/2)^{r/8}=3^{-\frac{r\log 2}{8\log 3}}\leq 2^{\frac{ \log 2}{8\log 3}}N^{-\frac{ \log 2}{8\log 3}}\leq 2^{\frac{ \log 2}{8\log 3}}N^{-0.078}.$$ In the penultimate inequality we used \eqref{r equation}. 

We now bound the cardinality of $S_{2}$. Equation \eqref{Bijection} tells us that the set of words of the form $a_{m+1}(n)\ldots a_{m+r}(n)$ for some $0\leq n<3^r$ coincides with $\{0,1,2\}^r.$ As such we can apply a well known large deviation inequality due to Hoeffding \cite{Hoe} to assert that 
\begin{align*}
\#S_{2}=\# \left\{b_1\ldots b_r\in \{0,1,2\}^{r}:\#\{1\leq j\leq r:b_{j}=1\}<\frac{r}{8}\right\}&\leq 2\cdot 3^r\cdot e^{-25r/288}\\
&=2\cdot 3^{r}\cdot 3^{\frac{-25r}{288\log 3}}\\
&=2\cdot 3^{r(1-\frac{25}{288\log 3})}\\
&\leq 2\cdot \left(\frac{3}{2}\right)^{1-\frac{25}{288\log 3}} N^{1-\frac{25}{288\log 3}}\\
&\leq 2\cdot \left(\frac{3}{2}\right)^{1-\frac{25}{288\log 3}} N^{0.922}
\end{align*} \eqref{WTS1} now follows upon taking $C_{1}=2^{\frac{ \log 2}{8\log 3}}$, $C_{2}=2\cdot \left(\frac{3}{2}\right)^{1-\frac{25}{288\log 3}}$, and observing that the cardinality of $S_{2}$ is an upper bounds for the cardinality of the set appearing on the left hand side of \eqref{WTS1}.
\end{proof}

\subsection{Integral estimates}
Given $a,b>0$ such that $a<b$ and a sequence $\mathbf{x}\in(\mathbb{R}/\mathbb{Z})^{\mathbb{N}},$ we let $(f_n)_{n=1}^{\infty}$ be a sequence of $C^2$ functions mapping the real numbers to the real numbers such that the following properties are satisfied:
\begin{enumerate}
	\item Each $f_n$ is $\mathbb{Z}$ periodic.
	\item $0\leq f_{n}(x)\leq 1$ for every $x\in\mathbb{R}$ and $n\in \mathbb{N}$.
	\item $f_{n}(x)=1$ for $x\in [x_n-an^{-0.01},x_n+an^{0.01}]+\mathbb{Z}.$
	\item $f_{n}(x)=0$ for $x\notin [x_n-bn^{-0.01},x_n+bn^{0.01}]+\mathbb{Z}.$
	\item $\|f'_{n}\|_{\infty}\ll n^{0.01}.$
	\item $\|f''_{n}\|_{\infty}\ll n^{0.02}.$
\end{enumerate}
We emphasise that the underlying constants appearing in $5.$ and $6.$ do depend upon $a$ and $b$. However this dependence won't influence our analysis so we suppress it from our notation. The bound on the second derivative of $f_n$ coming from property $6,$ together with the mean value theorem, tells us that we can take a modulus of continuity $\omega$ for $f_n'$ so that $\omega(x)\ll n^{0.02}|x|$. Now using Lemma \ref{Jackson's inequality}, we see that for any $N\in\mathbb{N}$ we have  
\begin{equation}
\label{errorbound}
\sum_{l\notin [-N,N]}c_{l,n}e^{2\pi ilx}=\O\left(\frac{n^{0.02}\log N}{N^{2}}\right).
\end{equation}
In \eqref{errorbound} and in what follows, for each $n\in\mathbb{N}$ we let $(c_{l,n})_{l\in\mathbb{Z}}$ denote the Fourier coefficients of $f_n$. Throughout this section we will suppress the dependence of $(f_n)_{n=1}^{\infty}$ and $(c_{l,n})_{l\in\mathbb{Z}}$ upon $a,$ $b,$ and $\mathbf{x}$ from our notation. The following two lemmas estimate the expectation and the variance of the random variable $\sum_{n=1}^{N}f_{n}(2^nx)$ for $x$ distributed according to $\mu$.
\begin{lemma}
	\label{Lemma1}
Let $(f_n)_{n=1}^{\infty}$ be a sequence of functions satisfying properties $1-6$ for some $a,b$ and $\mathbf{x}$. Then for any $N\in\mathbb{N}$ we have 
$$\int \sum_{n=1}^{N}f_{n}(2^nx)\, d\mu = \sum_{n=1}^{N}c_{0,n}+\O\left(N^{0.96}\log N\right)$$ and $$\int \sum_{n=1}^{N}f_{n}(2^nx)\, d\mu \asymp N^{0.99}.$$
\end{lemma}

\begin{proof}[Proof of Lemma \ref{Lemma1}]
	Fix $N\in\mathbb{N}$. Using \eqref{errorbound} and the fact that each $f_n$ coincides with its series, we have  $$\int \sum_{n=1}^{N}f_{n}(2^nx)\, d\mu = \sum_{n=1}^{N}c_{0,n}+\int \sum_{n=1}^{N}\sum_{l\in [-\lfloor N^{0.03}\rfloor ,\lfloor N^{0.03}\rfloor]\setminus \{0\}}c_{l,n}e(l2^nx)\, d\mu+\O\left(N^{0.96}\log N\right).$$ So to complete our proof of the first part of the lemma, it suffices to show that
	\begin{equation}
	\label{RTS}
	\int \sum_{n=1}^{N}\sum_{l\in [-\lfloor N^{0.03}\rfloor ,\lfloor N^{0.03}\rfloor]\setminus \{0\}}c_{l,n}e(l2^nx)\, d\mu=\O\left(N^{0.96}\log N\right).
	\end{equation}
We have 
\begin{equation}
\label{Lazy rewrite}
\int \sum_{n=1}^{N}\sum_{l\in [-\lfloor N^{0.03}\rfloor,\lfloor N^{0.03}\rfloor]\setminus \{0\}}c_{l,n}e(l2^nx)\, d\mu=\sum_{l\in [-\lfloor N^{0.03}\rfloor,\lfloor N^{0.03}\rfloor]\setminus \{0\}}\sum_{n=1}^{N}c_{l,n}\int e(l2^nx)\, d\mu.
\end{equation} Now let $C_{2}$ be as in Lemma \ref{Casselslemma}. Using this lemma and the fact that $c_{l,n}\ll n^{-0.01}$ for all $l\in\mathbb{Z},$ we have the following for any $l\in [-\lfloor N^{0.03}\rfloor,\lfloor N^{0.03}\rfloor]\setminus \{0\}:$
	\begin{align*}
	\left|\sum_{n=1}^{N}c_{l,n}\int e(l2^nx)\, d\mu\right|&\ll \sum_{n=1}^{\lfloor C_{2}N^{0.922}\rfloor }\frac{1}{n^{0.01}}+\sum_{n=1}^{N}\frac{N^{-0.078}}{n^{0.01}}\\
	&\ll N^{0.922\cdot 0.99}+N^{0.99-0.078}\\
	&\ll N^{0.922}.
	\end{align*}
	Applying this upper bound we see that 
	$$\left|\sum_{l\in [-\lfloor N^{0.03}\rfloor,\lfloor N^{0.03}\rfloor]\setminus \{0\}}\sum_{n=1}^{N}c_{l,n}\int e(l2^nx)\, d\mu\right|=\O(N^{0.952}).$$ Therefore by \eqref{Lazy rewrite} we see that \eqref{RTS} holds and our proof of the first part of this lemma is complete. 
	
	The second part of this lemma follows from the first part and the fact that $c_{0,n}\asymp n^{-0.01}$ for all $n.$  
\end{proof}

\begin{lemma}
	\label{Lemma2}
Let $(f_n)_{n=1}^{\infty}$ be a sequence of functions satisfying properties $1-6$ for some $a,b$ and $\mathbf{x}$. For any $N\in\mathbb{N}$ we have 
$$\int \left(\sum_{n=1}^{N}f_{n}(2^nx)-\int \sum_{n=1}^{N}f_{n}(2^ny)\, d\mu(y) \right)^2\, d\mu(x)=\O(N^{1.977}\log N).$$
\end{lemma}


	



\begin{proof}[Proof of Lemma \ref{Lemma2}]
Fix $N\in\mathbb{N}$. Multiplying out the bracket we have that 
\begin{align}
\label{Anna}
\int \left(\sum_{n=1}^{N}f_{n}(2^nx)-\int \sum_{n=1}^{N}f_{n}(2^ny)\, d\mu(y) \right)^2\,d\mu=\int \left( \sum_{n=1}^{N}f_{n}(2^nx)\right)^2 d\mu-\left( \int \sum_{n=1}^{N}f_{n}(2^nx)\, d\mu\right)^2.
\end{align}
 Focusing on the first term on the right hand side of \eqref{Anna}, we can use \eqref{errorbound} to assert that
\begin{align}
\label{Anna1}
&\int \left( \sum_{n=1}^{N}f_{n}(2^nx)\right)^2 d\mu\nonumber \\
=&\int \left(\sum_{n=1}^{N}c_{0,n}+\sum_{n=1}^{N}\sum_{l\in [-\lfloor N^{0.033}\rfloor,\lfloor N^{0.033}\rfloor]\setminus \{0\}}c_{l,n}e(l2^nx)+\O\left(N^{0.954}\log N\right)\right)^{2}\, d\mu.
\end{align} We want to eventually multiply out the bracket in \eqref{Anna1} and individually analyse the remaining terms. Before doing that we collect some straightforward bounds:
\begin{equation}
\label{trivial1}
\sum_{n=1}^{N}c_{0,n}\ll N^{0.99}
\end{equation}
and
\begin{equation}
\label{trivial2}
\left|\sum_{n=1}^{N}\sum_{l\in [-\lfloor N^{0.033}\rfloor,\lfloor N^{0.033}\rfloor]\setminus \{0\}}c_{l,n}e(l2^nx)\right|\ll \sum_{l\in [-\lfloor N^{0.033}\rfloor,\lfloor N^{0.033}\rfloor]\setminus \{0\}}\sum_{n=1}^{N}n^{-0.01}\ll N^{1.023}.
\end{equation}
\eqref{trivial1} follows from Lemma \ref{Lemma1}. \eqref{trivial2} uses the fact that $c_{l,n}\ll n^{-0.01}$ for all $l$ and $n$. Using \eqref{trivial1} and \eqref{trivial2}, we see that if we multiply out the bracket in \eqref{Anna1} we obtain the following:
\begin{align*}
&\int \left( \sum_{n=1}^{N}f_{n}(2^nx)\right)^2 d\mu\\
=&\underbrace{\left(\sum_{n=1}^{N}c_{0,n}\right)^{2}}_{A}+\underbrace{2\sum_{n=1}^{N}c_{0,n}\int \sum_{n=1}^{N}\sum_{l\in [-\lfloor N^{0.033}\rfloor,\lfloor N^{0.033}\rfloor]\setminus \{0\}}c_{l,n}e(l2^nx)\, d\mu}_{B}\\
&+\underbrace{\int \sum_{n=1}^{N}\sum_{l\in [-\lfloor N^{0.033}\rfloor,\lfloor N^{0.033}\rfloor]\setminus\{0\}}\sum_{m=1}^{N}\sum_{j\in [-\lfloor N^{0.033}\rfloor,\lfloor N^{0.033}\rfloor]\setminus\{0\}}c_{l,n}d_{j,m}e((l2^n+j2^m)x)\, d\mu}_{C}\\
&+\O(N^{1.944}\log N)+\O(N^{1.908}(\log N)^2)+\O(N^{1.977}\log N) .
\end{align*} These final three error terms are all $\O(N^{1.977}\log N)$. Therefore it remains to consider terms $A$, $B$, and $C$. By Lemma \ref{Lemma1} we know that 
$$\left(\int \sum_{n=1}^{N}f_{n}(2^nx)\, d\mu\right)^{2} = \left(\sum_{n=1}^{N}c_{0,n}\right)^{2}+\O(N^{1.95}\log N).$$ Therefore if we subtract the second term on the right hand side of \eqref{Anna} from term $A$ we have 
$$\left(\sum_{n=1}^{N}c_{0,n}\right)^{2}-\left( \int \sum_{n=1}^{N}f_{n}(2^nx)\, d\mu\right)^2=\O(N^{1.95}\log N).$$
It follows that to complete our proof we need to show that the two remaining terms $B$ and $C$ are $\O(N^{1.977}\log N)$. This we do below.

Duplicating the analysis given in the proof of Lemma \ref{Lemma1}, we can show that $$\left|\int \sum_{n=1}^{N}\sum_{l\in [-\lfloor N^{0.033}\rfloor,\lfloor N^{0.033}\rfloor]\setminus \{0\}}c_{l,n}e(l2^nx)\, d\mu\right|=O\left(N^{0.955}\right).$$ Now using Lemma \ref{Lemma1} it follows that
$$\left|\sum_{n=1}^{N}c_{0,n}\int \sum_{n=1}^{N}\sum_{l\in [-\lfloor N^{0.033}\rfloor,\lfloor N^{0.033}\rfloor]\setminus \{0\}}c_{l,n}e(l2^nx)\, d\mu\right|=\O(  N^{1.945}).$$
Therefore term $B$ is $\O(N^{1.977}\log N).$

Now we focus on term $C$. By considering the cases when $m=n$ and $m\neq n,$ we see that we can write term $C$ as follows:
\begin{align*}
&\int \sum_{n=1}^{N}\sum_{l\in [-\lfloor N^{0.033}\rfloor,\lfloor N^{0.033}\rfloor]\setminus\{0\}}\sum_{m=1}^{N}\sum_{j\in [-\lfloor N^{0.033}\rfloor,\lfloor N^{0.033}\rfloor]\setminus\{0\}}c_{l,n}d_{j,m}e((l2^n+j2^m)x)\, d\mu\\
=&\underbrace{\int\sum_{n=1}^{N}\sum_{l\in [-\lfloor N^{0.033}\rfloor,\lfloor N^{0.033}\rfloor]\setminus\{0\}}\sum_{j\in [-\lfloor N^{0.033}\rfloor,\lfloor N^{0.033}\rfloor]\setminus\{0\}}c_{l,n}d_{j,n}e((l2^n+j2^n)x)\, d\mu}_{D}\\
&+2\underbrace{\int \sum_{n=1}^{N-1}\sum_{l\in [-\lfloor N^{0.033}\rfloor,\lfloor N^{0.033}\rfloor]\setminus\{0\}}\sum_{m=n+1}^{N}\sum_{j\in [-\lfloor N^{0.033}\rfloor,\lfloor N^{0.033}\rfloor]\setminus\{0\}}c_{l,n}d_{j,m}e((l2^n+j2^m)x)\, d\mu}_{E}.
\end{align*}
Taking the trivial upper bound of $1$ for the absolute value of the terms appearing in $D,$ we obtain the following upper bound
$$\int\sum_{n=1}^{N}\sum_{l\in [-\lfloor N^{0.033}\rfloor,\lfloor N^{0.033}\rfloor]\setminus\{0\}}\sum_{j\in [-\lfloor N^{0.033}\rfloor,\lfloor N^{0.033}\rfloor]\setminus\{0\}}c_{l,n}d_{j,n}e((l2^n+j2^n)x)\, d\mu=\O(N^{1.066}).$$ Therefore term $D$ is $\O(N^{1.977}\log N)$ and it now suffices to bound term $E$. We start by bounding term $E$ from above by an expression to which Lemma \ref{Casselslemma} can be applied:
\begin{align*}
&\int \sum_{n=1}^{N-1}\sum_{l\in [-\lfloor N^{0.033}\rfloor,\lfloor N^{0.033}\rfloor]\setminus\{0\}}\sum_{m=n+1}^{N}\sum_{j\in [-\lfloor N^{0.033}\rfloor,\lfloor N^{0.033}\rfloor]\setminus\{0\}}c_{l,n}d_{j,m}e((l2^n+j2^m)x)\, d\mu\\
=&\int \sum_{n=1}^{N-1}\sum_{l\in [-\lfloor N^{0.033}\rfloor,\lfloor N^{0.033}\rfloor]\setminus\{0\}}\sum_{m=n+1}^{N}\sum_{j\in [-\lfloor N^{0.033}\rfloor,\lfloor N^{0.033}\rfloor]\setminus\{0\}}c_{l,n}d_{j,m}e(2^{n}(l+j2^{m-n})x)\, d\mu\\
\ll & \sum_{n=1}^{N-1}\sum_{l\in [-\lfloor N^{0.033}\rfloor,\lfloor N^{0.033}\rfloor]\setminus\{0\}}\sum_{k=1}^{N-n}\sum_{j\in [-\lfloor N^{0.033}\rfloor,\lfloor N^{0.033}\rfloor]\setminus\{0\}}\frac{1}{n^{0.02}}\left|\int e(2^{n}(l+j2^k)x)\, d\mu\right|\\
\ll & \sum_{n=1}^{N-1}\sum_{l\in [-\lfloor N^{0.033}\rfloor,\lfloor N^{0.033}\rfloor]\setminus\{0\}}\sum_{k=1}^{N}\sum_{j\in [-\lfloor N^{0.033}\rfloor,\lfloor N^{0.033}\rfloor]\setminus\{0\}}\frac{1}{n^{0.02}}\left|\int e(2^{n}(l+j2^k)x)\, d\mu\right|\\
= & \sum_{l\in [-\lfloor N^{0.033}\rfloor,\lfloor N^{0.033}\rfloor]\setminus\{0\}}\sum_{k=1}^{N}\sum_{j\in [-\lfloor N^{0.033}\rfloor,\lfloor N^{0.033}\rfloor]\setminus\{0\}} \sum_{n=1}^{N-1}\frac{1}{n^{0.02}}\left|\int e(2^{n}(l+j2^k)x)\, d\mu\right|.
\end{align*}
In the third line in the above we have used that $c_{l,n}\ll n^{-0.01}$ for all $n$ and $d_{j,m}\ll n^{-0.01}$ for all $m>n$. To bound the final term in the above we have to be careful to manage those parameters for which $l+j2^k=0$. Rearranging this equation and taking logarithms, we see that if $k$ is such that there exists $l,j\in [-\lfloor N^{0.033}\rfloor,\lfloor N^{0.033}\rfloor]\setminus\{0\}$ for which $l+j2^{k}=0,$ then $k\leq \lfloor 0.033\log_{2}N \rfloor.$ With this condition in mind we split the final term above as follows: 
\begin{align*}
&\sum_{l\in [-\lfloor N^{0.033}\rfloor,\lfloor N^{0.033}\rfloor]\setminus\{0\}}\sum_{k=1}^{N}\sum_{j\in [-\lfloor N^{0.033}\rfloor,\lfloor N^{0.033}\rfloor]\setminus\{0\}} \sum_{n=1}^{N-1}\frac{1}{n^{0.02}}\left|\int e(2^{n}(l+j2^k)x)\, d\mu\right|\\
=&\underbrace{\sum_{l\in [-\lfloor N^{0.033}\rfloor,\lfloor N^{0.033}\rfloor]\setminus\{0\}}\sum_{k=1}^{\lfloor 0.033\log_{2}N \rfloor}\sum_{j\in [-\lfloor N^{0.033}\rfloor,\lfloor N^{0.033}\rfloor]\setminus\{0\}} \sum_{n=1}^{N-1}\frac{1}{n^{0.02}}\left|\int e(2^{n}(l+j2^k)x)\, d\mu\right|}_{F}\\
+& \underbrace{\sum_{l\in [-\lfloor N^{0.033}\rfloor,\lfloor N^{0.033}\rfloor]\setminus\{0\}}\sum_{k=\lfloor 0.033\log_{2}N \rfloor+1}^{N}\sum_{j\in [-\lfloor N^{0.033}\rfloor,\lfloor N^{0.033}\rfloor]\setminus\{0\}} \sum_{n=1}^{N-1}\frac{1}{n^{0.02}}\left|\int e(2^{n}(l+j2^k)x)\, d\mu\right|}_{G}.
\end{align*}To complete our proof it now suffices to show that term $F$ and term $G$ are both $\O(N^{1.977}\log N).$ Taking the trivial upper bound of $1$ for the absolute value of the terms appearing in the summation in $F$, we have 
\begin{align*}
\sum_{l\in [-\lfloor N^{0.033}\rfloor,\lfloor N^{0.033}\rfloor]\setminus\{0\}}\!\sum_{k=1}^{\lfloor 0.033\log_{2}N \rfloor}\sum_{j\in [-\lfloor N^{0.033}\rfloor,\lfloor N^{0.033}\rfloor]\setminus\{0\}} \!\sum_{n=1}^{N-1}\frac{1}{n^{0.02}}\left|\int e(2^{n}(l+j2^k)x)\, d\mu\right|=\O(N^{1.066}\log N).
\end{align*}
So term $F$ is $\O(N^{1.977}\log N)$. Now we focus on term $G$. If $l+j2^{k}\neq 0$ then we can apply Lemma \ref{Casselslemma} to assert the following: 
\begin{align*}
\sum_{n=1}^{N-1}\frac{1}{n^{0.02}}\left|\int e(2^{n}(l+j2^k)x)\, d\mu\right|&\ll \sum_{n=1}^{\lfloor C_{2} N^{0.922}\rfloor }\frac{1}{n^{0.02}}+ \sum_{n=1}^{N}\frac{N^{-0.078}}{n^{0.02}}\\
&\ll N^{0.922\cdot 0.98}+N^{0.98-0.078}\\
&\ll N^{0.904}.
\end{align*}
Because each parameter $k$ appearing in term $G$ satisfies $k\geq\lfloor 0.033\log_{2}N \rfloor+1,$ we know that $l+j2^k\neq 0$ for all $l,j\in [-\lfloor N^{0.033}\rfloor,\lfloor N^{0.033}\rfloor]\setminus\{0\}$. Therefore we can freely apply the bound above and the following holds: 
\begin{align*}
\sum_{l\in [-\lfloor N^{0.033}\rfloor,\lfloor N^{0.033}\rfloor]\setminus\{0\}}&\sum_{k=\lfloor 0.033\log_{2}N \rfloor+1}^{N}\sum_{j\in [-\lfloor N^{0.033}\rfloor,\lfloor N^{0.033}\rfloor]\setminus\{0\}} \sum_{n=1}^{N-1}\frac{1}{n^{0.02}}\left|\int e(2^{n}(l+j2^k)x)\, d\mu\right|\\
\ll &\sum_{l\in [-\lfloor N^{0.033}\rfloor,\lfloor N^{0.033}\rfloor]\setminus\{0\}}\sum_{k=\lfloor 0.033\log_{2}N \rfloor+1}^{N}\sum_{j\in [-\lfloor N^{0.033}\rfloor,\lfloor N^{0.033}\rfloor]\setminus\{0\}} N^{0.904}\\
\ll &N^{1.97}.
\end{align*}
Therefore term $G$ is also $\O(N^{1.977}\log N)$. This completes our proof.  
\end{proof}

\subsection{Completing the proof of Theorem \ref{Main thm}}

Equipped with Lemma \ref{Lemma1} and Lemma \ref{Lemma2} we are now in a position to prove Theorem \ref{Main thm}.

\begin{proof}[Proof of Theorem \ref{Main thm}]
Fix $\mathbf{x}\in(\mathbb{R}/\mathbb{Z})^{\mathbb{N}}.$ Let $\epsilon>0$ be arbitrary and $(f_n)_{n=1}^{\infty}$ be a sequence of $C^2$ functions satisfying properties $1-6$ for $a=1,$ $b=1+\epsilon,$ and $\mathbf{x}$. It follows from these properties and this choice of parameters that we have the following upper bound for any $x\in C$
\begin{equation}
\label{Counting upper bound}
\#\{1\leq n\leq N:d(T_{2}^{n}(x),x_n)\leq n^{-0.01}\}\leq \sum_{n=1}^{N}f_{n}(2^{n}x).
\end{equation}
By our choice of $a$ and $b$ it is also clear that 
\begin{equation}
\label{abequation}
\sum_{n=1}^{N}c_{0,n}\leq 2(1+\epsilon)\sum_{n=1}^{N}n^{-0.01}.
\end{equation}
By Lemma \ref{Lemma2} and Markov's inequality we have the following for any $N\in\mathbb{N}$
\begin{align*}
&\mu\left(x\in C:\left|\sum_{n=1}^{N^{1000}}f_{n}(2^{n}x)-\int \sum_{n=1}^{N^{1000}}f_{n}(2^ny)\, d\mu(y)\right|\geq N^{989.9}\right)\\
=&\mu\left(x\in C:\left(\sum_{n=1}^{N^{1000}}f_{n}(2^{n}x)-\int \sum_{n=1}^{N^{1000}}f_{n}(2^ny)\, d\mu(y)\right)^{2}\geq N^{1979.8}\right)\\
\ll& \frac{N^{1977}\log N}{N^{1979.8}}\\
=&\frac{\log N}{N^{2.8}}.
\end{align*}
Using this bound we see that
$$\sum_{N=1}^{\infty}\mu\left(x\in C:\left|\sum_{n=1}^{N^{1000}}f_{n}(2^{n}x)-\int \sum_{n=1}^{N^{1000}}f_{n}(2^ny)\, d\mu(y)\right|\geq N^{989.9}\right)<\infty.$$
Therefore by the Borel-Cantelli lemma, for $\mu$ almost every $x$ there exists finitely many $N\in\mathbb{N}$ for which 	
\begin{equation}
\label{finitesols}
\left|\sum_{n=1}^{N^{1000}}f_{n}(2^{n}x)-\int \sum_{n=1}^{N^{1000}}f_{n}(2^ny)\, d\mu(y)\right|\geq N^{989.9}.
\end{equation} 
Given an arbitrary $N\in\mathbb{N}$ we now define $K_{N}\in\mathbb{N}$ according to the inequalities $$K_{N}^{1000}\leq N< (K_{N}+1)^{1000}.$$ 
For later use we record here the following straightforward facts
\begin{equation}
\label{Kasymptotics}\lim_{K\to\infty}\frac{\sum_{n=1}^{(K+1)^{1000}}n^{-0.01}}{\sum_{n=1}^{K^{1000}}n^{-0.01}}=1
\end{equation}and 
\begin{equation}
\label{Kgrowth}
\sum_{n=1}^{K^{1000}}n^{-0.01}\asymp K^{990}.
\end{equation}

Now let $x$ belong to the full $\mu$ measure set for which \eqref{finitesols} is satisfied by finitely many $N$. Then we have
\begin{align*}
&\limsup_{N\to\infty}\frac{\{1\leq n\leq N:d(T_{2}^{n}(x),x_n)\leq n^{-0.01}\}}{2\sum_{n=1}^{N}n^{-0.01}}\\
\stackrel{\eqref{Counting upper bound}}{\leq} & \limsup_{N\to\infty}\frac{\sum_{n=1}^{N}f_{n}(2^nx)}{2\sum_{n=1}^{N}n^{-0.01}}\\
\leq & \limsup_{N\to\infty}\frac{\sum_{n=1}^{(K_{N}+1)^{1000}}f_{n}(2^nx)}{2\sum_{n=1}^{K_{N}^{1000}}n^{-0.01}}\\
\stackrel{\eqref{finitesols}}{\leq}& \limsup_{N\to\infty}\frac{\int \sum_{n=1}^{(K_{N}+1)^{1000}}f_{n}(2^ny)\, d\mu(y)+(K_{N}+1)^{989.9}}{2\sum_{n=1}^{K_{N}^{1000}}n^{-0.01}}\\
\stackrel{\text{Lemma }\ref{Lemma1}}{=}&\limsup_{N\to\infty}\frac{\sum_{n=1}^{(K_{N}+1)^{1000}}c_{0,n}+(K_{N}+1)^{989.9}+\O(K_{N}^{960}\log K_{N})}{2\sum_{n=1}^{K_{N}^{1000}}n^{-0.01}}\\
\stackrel{\eqref{abequation}}{\leq}& \limsup_{N\to\infty}\frac{2(1+\epsilon)\sum_{n=1}^{(K_N+1)^{1000}}n^{-0.01}+(K_{N}+1)^{989.9}+\O(K_{N}^{960}\log K_{N})}{2\sum_{n=1}^{K_{N}^{1000}}n^{-0.01}}\\
\stackrel{\eqref{Kasymptotics},\eqref{Kgrowth}}{=}&1+\epsilon.
\end{align*}
Summarising the above, we have shown that for $\mu$ almost every $x$ we have 
$$\limsup_{N\to\infty}\frac{\{1\leq n\leq N:d(T_{2}^{n}(x),x_n)\leq n^{-0.01}\}}{2\sum_{n=1}^{N}n^{-0.01}}\leq 1+\epsilon.$$ Since $\epsilon$ is arbitrary we may conclude that 
\begin{equation}
\label{limsup}
\limsup_{N\to\infty}\frac{\{1\leq n\leq N:d(T_{2}^{n}(x),x_n)\leq n^{-0.01}\}}{2\sum_{n=1}^{N}n^{-0.01}}\leq 1
\end{equation}
for $\mu$ almost every $x$. 

By an analogous argument, this time taking a sequence of functions $(f_n)_{n=1}^{\infty}$ satisfying properties $1-6$ for $a=1-\epsilon$ and $b=1,$ we can show that for $\mu$ almost every $x$ we have 
$$\liminf_{N\to\infty}\frac{\{1\leq n\leq N:d(T_{2}^{n}(x),x_n)\leq n^{-0.01}\}}{2\sum_{n=1}^{N}n^{-0.01}}\geq 1-\epsilon.$$ Because $\epsilon$ is arbitrary, we have that
\begin{equation}
\label{liminf}
\liminf_{N\to\infty}\frac{\{1\leq n\leq N:d(T_{2}^{n}(x),x_n)\leq n^{-0.01}\}}{2\sum_{n=1}^{N}n^{-0.01}}\geq 1
\end{equation}
for $\mu$ almost every $x$. Therefore \eqref{limsup} and \eqref{liminf} hold for $\mu$ almost every $x$. This completes our proof.

\end{proof}


\begin{thebibliography}{100}
\bibitem{ABS} A. Algom, S. Baker, P. Shmerkin, \textit{ On normal numbers and self-similar measures,} Adv. Math. (to appear).
\bibitem{ACY} D. Allen, S. Chow, H. Yu, \textit{Dyadic Approximation in the Middle-Third Cantor Set,} arXiv:2005.09300.
\bibitem{Bak} S. Baker, \textit{Intrinsic Diophantine Approximation for overlapping iterated function systems,} arXiv:2104.14249.
\bibitem{BFR} R. Broderick, L. Fishman, A, Reich, \textit{Intrinsic approximation on Cantor-like sets, a problem of Mahler,} Mosc. J. Comb. Number Theory 1 (2011), no. 4, 3--12.
\bibitem{Bug2} Y. Bugeaud, \textit{Diophantine approximation and Cantor sets,}
Math. Ann. 341 (2008), no. 3, 677--684.
\bibitem{Bug3} Y. Bugeaud, \textit{Distribution modulo one and Diophantine approximation,} Cambridge Tracts in Mathematics, 193, Cambridge University Press, Cambridge, 2012.

\bibitem{BugDur} Y. Bugeaud, A. Durand, \textit{Metric Diophantine approximation on the middle-third Cantor set,}
J. Eur. Math. Soc. (JEMS) 18 (2016), no. 6, 1233--1272.

\bibitem{Cas} J. Cassels, \textit{On a problem of Steinhaus about normal numbers,} Colloq. Math., 7:95–-101,
1959. 

\bibitem{Fal} K. Falconer, \textit{Fractal geometry. 
	Mathematical foundations and applications,} Third edition. John Wiley \& Sons, Ltd., Chichester, 2014. xxx+368 pp. ISBN: 978-1-119-94239-9. 
\bibitem{FMS} L. Fishman, K. Merrill, D. Simmons, \textit{Uniformly de Bruijn sequences and symbolic Diophantine approximation on fractals,} Ann. Comb. 22 (2018), no. 2, 271--293.
\bibitem{FS2} L. Fishman, D. Simmons, \textit{Extrinsic Diophantine approximation on manifolds and fractals,} J. Math. Pures Appl. (9) 104 (2015), no. 1, 83--101.
\bibitem{FS} L. Fishman, D. Simmons, \textit{Intrinsic approximation for fractals defined by rational iterated function systems: Mahler's research suggestion,} Proc. Lond. Math. Soc. (3) 109 (2014), no. 1, 189--212.
\bibitem{Fur} H. Furstenberg, \textit{Disjointness in Ergodic Theory, Minimal Sets, and a Problem in Diophantine Approximation,} Mathematical systems theory, 1(1),(1967), 1--49.
\bibitem{Hoc} M. Hochman, \textit{A short proof of Host's equidistribution theorem,} arXiv:2103.08938.
\bibitem{HocShm} M. Hochman, P. Shmerkin, \textit{Equidistribution from fractal measures,} Inventiones mathematicae, 202(1):427--479, 2015.
\bibitem{HocShm2} M. Hochman, P. Shmerkin, \textit{Local entropy averages and projections of fractal measures,} Ann. Math., {\bf 175}, (2012), 1001--1059.
\bibitem{Hoe} W. Hoeffding, \emph{Probability inequalities for sums of bounded random variables,} J. Amer. Statist. Assoc. 58 1963 13–-30.
\bibitem{Host}B. Host. Nombres normaux, entropie, translations. Israel J. Math., 91(1-3):419--428, 1995.
\bibitem{Jac} D. Jackson, \textit{The theory of approximation,} American Mathematical Society Colloquium Publications, 11. American Mathematical Society, Providence, RI, 1994. viii+178 pp. ISBN: 0-8218-1011-1.
\bibitem{Joh} A. S. A. Johnson. \textit{Measures on the circle invariant under multiplication by a nonlacunary subsemigroup of the integers} Israel J. Math., 77(1-2):211--240, 1992.
\bibitem{KL} O. Khalil, M. Luethi, \textit{Random Walks, Spectral Gaps, and Khintchine's Theorem on Fractals,} arXiv:2101.05797. 
\bibitem{Khi} A. Khintchine, \textit{Einige S\"atze \"uber Kettenbruche, mit Anwendungen auf die Theorie der Diophantischen Approximationen,} Math. Ann. 92
(1924), 115--125.
\bibitem{KLW} D. Kleinbock, E. Lindenstrauss, B. Weiss, \textit{On fractal measures and Diophantine approximation,} Selecta Math. (N.S.) 10 (2004), no. 4, 479--523.

\bibitem{LSV} J. Levesley, C. Salp, S. Velani, \textit{On a problem of K. Mahler: Diophantine approximation and Cantor sets,} Math. Ann. 338 (2007), no. 1, 97--118.
\bibitem{Lin}  E, Lindenstrauss, \textit{p-adic foliation and equidistribution,} Israel J. Math., 122:29--42, 2001.
\bibitem{Mah} K. Mahler, \textit{Some suggestions for further research,}
Bull. Austral. Math. Soc. 29 (1984), no. 1, 101--108.
\bibitem{Phi} W. Philipp, \textit{Some metrical theorems in number theory,} Pacific J. Math. 20 (1967), 109--127.
\bibitem{Rud}  D. Rudolph. \textit{$\times 2$ and $\times 3$ invariant measures and entropy,} Ergodic Theory Dynam. Systems, 10(2):395--406, 1990.
\bibitem{Sch} J. Schleischitz, \textit{On intrinsic and extrinsic rational approximation to Cantor sets,} Ergod. Th. Dynam. Sys. 41 (2021), no. 5, 1560--1589.
\bibitem{Shmerkin} P. Shmerkin, \textit{On Furstenberg's intersection conjecture, self-similar measures, and the $L^q$ norms of convolutions,} 
Ann. of Math. (2) 189 (2019), no. 2, 319–-391.
\bibitem{Shp} I. Shparlinski, \textit{On the arithmetic structure of rational numbers in the Cantor set,} Bull. Aust. Math. Soc., 2021, v.103, 22--27.
\bibitem{SW} D. Simmons, B. Weiss, \textit{Random walks on homogeneous spaces and Diophantine approximation on fractals,} Invent. Math. 216 (2019), no. 2, 337--394.
\bibitem{TWW} B. Tan, B. Wang, J. Wu, \textit{Mahler's question for intrinsic Diophantine approximation on triadic Cantor set: the divergence theory,} 	arXiv:2103.00544.
\bibitem{Weiss} B. Weiss, \textit{Almost no points on a Cantor set are very well approximable,} R. Soc. Lond. Proc. Ser. A Math. Phys. Eng. Sci. 457 (2001), no. 2008, 949--952.
\bibitem{Wu} M. Wu,  \textit{A proof of Furstenberg’s conjecture on the intersections of xp and xq-invariant sets,} Ann. of Math., (2) 189 no. 3 (2019), 707--751.
\bibitem{Yu} H. Yu, \textit{Rational points near self-similar sets,} 	arXiv:2101.05910. 
\end{thebibliography}
\end{document}